\documentclass[12pt,a4paper]{article}
\usepackage{amsfonts}
\usepackage{amsmath}
\usepackage{amsthm, amscd}
\usepackage{mathtools}
\usepackage{amssymb}

\numberwithin{equation}{section}

\theoremstyle{plain}\newtheorem{Theorem}{Theorem}[section]
\theoremstyle{plain}\newtheorem{Corollary}[Theorem]{Corollary}
\theoremstyle{plain}\newtheorem{Lemma}[Theorem]{Lemma}
\theoremstyle{plain}\newtheorem{Definition}[Theorem]{Definition}
\theoremstyle{plain}\newtheorem{Proposition}[Theorem]{Proposition}
\theoremstyle{plain}\newtheorem*{Theorem*}{Theorem}
\theoremstyle{remark}

\DeclareMathOperator{\Hom}{Hom}
\DeclareMathOperator{\Ima}{Im}
\DeclareMathOperator{\ind}{ind}

\title{Modulo 2 counting of Klein-bottle leaves in smooth taut foliations}
\author{Boyu Zhang}
\date{}
\begin{document}
\maketitle
\abstract This article proves that the parity of the number of Klein-bottle leaves in a smooth cooriented taut foliation is invariant under smooth deformations within taut foliations, provided that every Klein-bottle leaf involved in the counting has non-trivial linear holonomy.

\section{Introduction} \label{sec: statement of result}
It was proved in \cite{bonatti1994feuilles} that for a cooriented foliation, a $C^0$-generic smooth perturbation destroys all closed leaves with genus greater than $1$. This article explores the other side of the story. It shows that under certain conditions, one cannot get rid of Klein-bottle leaves of a taut foliation by smooth deformations. 

Let $\mathcal{L}$ be a smooth cooriented 2-dimensional foliation on a smooth three manifold $Y$. The foliation $\mathcal{L}$ and the manifold $Y$ are allowed to be unorientable. By definition, the foliation $\mathcal{L}$ is called a taut foliation if for every point $p \in Y$ there exists an embedded circle in $Y$ passing through $p$ and being transverse to $\mathcal{L}$.

\begin{Definition}
Let $K\subset Y$ be a closed leaf of $\mathcal{L}$. The leaf $K$ is called nondegenerate if it has non-trivial linear holonomy.
\end{Definition}

Consider a closed 2-dimensional submanifold $K$ of $Y$. If $K$ is cooriented, one can define an element $PD[K]\in\Hom(H_1(Y;\mathbb{Z});\mathbb{Z})$ as follows. Let $[\gamma]$ be a homology class represented by a closed curve $\gamma$, then $PD[K]$ maps $[\gamma]$ to the oriented intersection number of $\gamma$ and $K$. Since $\Hom(H_1(Y;\mathbb{Z});\mathbb{Z})\cong H^1(Y;\mathbb{Z})$, the element $PD[K]$ can be considered as an element of $H^1(Y;\mathbb{Z})$. If both $Y$ and $K$ are oriented and if the orientations of $Y$ and $K$ are compatible with the coorientation of $K$, the element $PD[K]$ is equal to the Poincar\'e dual of the fundamental class of $K$.

\begin{Definition}
Let $A\in H^1(Y;\mathbb{Z})$. A closed leaf $K$ of $\mathcal{L}$ is said to be in the class $A$ if $PD[K]=A$. The foliation $\mathcal{L}$ is called $A$-admissible if every Klein-bottle leaf of $\mathcal{L}$ in the class $A$ is nondegenerate. 
\end{Definition}

The following result is the main theorem of this article.

\begin{Theorem}\label{thm: main}
Let $A\in H^1(Y;\mathbb{Z})$. Let $\mathcal{L}_s$, $s\in[0,1]$ be a smooth family of coorientable taut foliations on $Y$. Suppose $\mathcal{L}_0$ and $\mathcal{L}_1$ are both $A$-admissible. For $i=0,1$, let $n_i$ be the number of Klein-bottle leaves in the class $A$. Then $n_0$ and $n_1$ have the same parity.
\end{Theorem}

Notice that if there is no Klein-bottle leaf of $\mathcal{L}$ in the homology class $A$, then $\mathcal{L}$ is automatically $A$-admissible. Therefore, the following result follows immediately.

\begin{Corollary} \label{cor: existence of Klein bottle leaf}
Let $A\in H^1(Y;\mathbb{Z})$, and let $\mathcal{L}$ be an $A$-admissible smooth coorientable taut foliation on $Y$. Assume that $\mathcal{L}$ has an odd number of Klein-bottle leaves in the class $A$. Then every smooth deformation of $\mathcal{L}$ through taut foliations has at least one Klein-bottle leaf in the class $A$. \qed
\end{Corollary}

It would be interesting to understand whether a similar result holds for torus leaves of taut foliations. For example, suppose $\mathcal{L}_0$ and $\mathcal{L}_1$ are two oriented and cooriented taut foliations on $Y$ that can be deformed to each other through taut foliations. Suppose every closed torus leaf in a homology class $e\in H_2(Y;\mathbb{Z})$ has non-trivial linear holonomy, is it always true that the numbers of torus leaves in the homology class $e$ in $\mathcal{L}_0$ and $\mathcal{L}_1$ have the same parity? The answer to this question is not clear to the author at the time of writing this article.

The article is organized as follows. Sections \ref{sec: moduli spcae} and \ref{sec: symplectization} build up necessary tools for the proof of theorem \ref{thm: main}. Sections \ref{sec: proof} and \ref{sec: technical lemmas} prove the theorem. Section \ref{sec: an example} gives an explicit example of corollary \ref{cor: existence of Klein bottle leaf}, constructing a foliation with a Klein-bottle leaf that cannot be removed by deformations.

I would like to express my most sincere gratitude to Cliff Taubes, who has been consistently providing me with inspiration, encouragement, and patient guidance.

\section{Moduli spaces of $J$-holomorphic tori} \label{sec: moduli spcae}

In \cite{taubes1996counting}, Taubes studied the behaviour of the moduli space of pseudo-holomorphic curves on a compact symplectic 4-manifold, and used it to define a version of Gromov invariant. This section recalls some results from \cite{taubes1996counting} to prepare for the proof of theorem \ref{thm: main}. The moduli space considered here is not exactly the moduli space used in the definition of Taubes's Gromov invariant, but essentially what is developed in this section is a special case of Taubes's result. For a survey on different versions of Gromov invariants of symplectic 4-manifolds based on Taubes's work, see \cite{mcduff1997lectures}.

Let $X$ be a smooth 4-manifold. To avoid complications caused by exceptional spheres, assume throughout this section that $\pi_2(X)=0$. This will be enough for the proof of theorem \ref{thm: main}. Let $J$ be a smooth almost complex structure on $X$.

Consider an immersed closed $J$-holomorphic curve $C$ in $X$. Let $N$ be the normal bundle of $C$, the fiber of $N$ then inherits an almost complex structure from $J$. Let $\pi:N\to C$ be the projection from $N$ to $C$. Choose a local diffeomorphism $\varphi$ from a neighborhood of the zero section of $N$ to a neighborhood of $C$ in $X$, which maps the zero section of $N$ to $C$. 
The map $\varphi$ can be chosen in such a way that the tangent map is $\mathbb{C}$-linear on the zero section of $N$. Every closed immersed $J$-holomorphic curve that is $C^1$-close to $C$ is the image of a section of $N$. Fix an arbitrary connection $\nabla_0$ on $N$ and let $\bar{\partial}_0$ be the $(0,1)$-part of $\nabla_0$. 
If $s$ is a section of $N$ near the zero section, the equation for $\varphi(s)$ to be a $J$-holomorphic curve in $X$ can be schematically written as
\begin{equation} \label{eqn:J holomorphic for sections}
\bar\partial_0 s + \tau(s) (\nabla_0(s)) + Q(s) (\nabla_0(s),\nabla_0(s)) + T(s)=0.
\end{equation}
Here $\tau$ is a smooth section of $\pi^*(\Hom_{\mathbb{R}}(T^*C\otimes_\mathbb{R} N, T^{0,1}C\otimes_{\mathbb{C}}N))$, and $Q$ is a smooth section of $\pi^*(\Hom_{\mathbb{R}}(T^*C\otimes_\mathbb{R} N\otimes_\mathbb{R} N, T^{0,1}C\otimes_{\mathbb{C}}N))$, and $T$ is a smooth section of $\pi^*(T^{0,1}C\otimes_{\mathbb{C}} N)$. The values of $\tau$, $Q$, and $T$ are defined pointwise by the values of $J$ in an algebraic way, and $\tau$, $Q$, $T$ are zero when $s=0$. The linearized equation of \eqref{eqn:J holomorphic for sections} at $s=0$ is
$\bar\partial_0(s)+\frac{\partial{T}}{\partial{s}}(s)=0.$
Define 
\begin{equation} \label{eqn: deformation operator}
L(s):=\bar\partial_0(s)+\frac{\partial{T}}{\partial{s}}(s).
\end{equation}
Notice that $L$ is a real linear operator. The curve $C$ is called nondegenerate if $L$ is surjective as a map from $L_1^2(N)$ to $L^2(N)$. By elliptic regularity, if $C$ is nondegenerate then the operator $L$ is surjective as a map from $L_k^2(N)$ to $L_{k-1}^2(N)$ for every $k\ge 1$. The index of the operator $L$ equals 
\begin{equation} \label{eqn: index}
\ind L=\langle c_1(N), [C] \rangle-\langle c_1(T^{0,1}X), [C]\rangle.
\end{equation}

It follows from the definition that nondegeneracy only depends on the 1-jet of $J$ on $C$. Namely, if there is another almost complex structure $J'$ such that $(J-J')|_C=0$ and $(\nabla(J-J'))|_C=0$, then $C$ is nondegenerate as a $J$-holomorpic curve if and only if it is nondegenerate as a $J'$-holomorphic curve.

For a homology class $e\in H_2(X;\mathbb{Z})$, define 
$$d(e)=e\cdot e - \langle c_1(T^{0,1}X), e\rangle.$$
By equation \eqref{eqn: index}, $d(e)$ is the formal dimension of the moduli space of embedded pseudo-holomorphic curves in $X$ in the homology class $e$. By the adjunction formula, the genus $g$ of such a curve satisfies 
$$e\cdot e + 2-2g = -\langle c_1(T^{0,1}X), e \rangle. $$
Therefore $d(e)=2(g-\langle c_1(T^{0,1}X),e \rangle -1)$. In general, the formal dimension of the moduli space of $J$-holomorphic \emph{maps} from a genus $g$ curve to $X$ in the homology class $e$, modulo self-isomorphisms of the domain, is equal to $2(g-\langle c_1(T^{0,1}X),e \rangle -1)$.

Now assume $X$ is has a symplectic structure $\omega$. Recall that an almost complex structure $J$ is compatible with $\omega$ if $\omega(\cdot,J\cdot)$ defines a Riemannian metric. Let $\mathcal{J}(X,\omega)$ be the set of smooth almost complex structures compatible with $\omega$. For a closed surface $\Sigma$ and a map $\rho:\Sigma\to X$, define the topological energy of $\rho$ to be $\int_\Sigma \, \rho^*(\omega)$. 

\begin{Definition}\label{def: admissible}
Let $(X,\omega)$ be a symplectic manifold. Let $E>0$ be a constant. An almost complex structure $J\in \mathcal{J}(X,\omega)$ is called $E$-admissible if the following conditions hold:
\begin{enumerate}
\item Every embedded $J$-holomorphic curve $C$ with energy less than or equal to $E$ and with $d([C])=0$ is nondegenerate.
\item  \label{condition 2} For every homology class $e\in H_2(X;\mathbb{Z})$, if $\langle [\omega], e\rangle \le E$, and if $\langle c_1(T^{0,1}X), e\rangle >0$ (namely, the formal dimension of the moduli space of $J$-holomorphic maps from a torus to $X$ in the homology class $e$, modulo self-isomorphisms of the domain, is negative), then there is no somewhere injective $J$-holomorphic map from a torus to $X$ in the homology class $e$.
\end{enumerate}
\end{Definition}

The next lemma is a special case of proposition 7.1 in \cite{taubes1996sw}. Recall that the $C^\infty$ topology on $\mathcal{J}(X,\omega)$ is defined as the Fr\'echet topology, namely it is induced by the distance function
$$
d(j_1,j_2)=\sum_{n=1}^{\infty} 2^{-n}\cdot\frac{\|j_1-j_2\|_{C^n}}{1+\|j_1-j_2\|_{C^n}}.
$$ 
\begin{Lemma}
Let $E>0$ be a constant. If $(X,\omega)$ is a compact symplectic manifold, the set of $E$-admissible almost complex structures form an open and dense subset of $\mathcal{J}(X,\omega)$ in the $C^\infty$-topology. \qed
\end{Lemma}

A homology class $e$ is called primitive if $e\neq n\cdot e'$ for every integer $n>1$ and every $e'\in H_2(X;\mathbb{Z})$.  
If $e\in H_2(X;\mathbb{Z})$ is a primitive class, define $\mathcal{M}(X,J,e)$ to be the set of embedded $J$-holomorphic tori in $X$ with fundamental class $e$.

Now consider smooth families of almost complex structures. Assume $\omega_s$ $(s\in[0,1])$ is a smooth family of symplectic forms on $X$. For $i=0,1$, let $J_i\in\mathcal{J}(X,\omega_i)$. Define
$$
\mathcal{J}(X,\{\omega_s\},J_0,J_1)
$$
to be the set of smooth families $\{J_s\}$ connecting $J_0$ and $J_1$, such that $J_s\in\mathcal{J}(X,\omega_s)$ for each $s\in[0,1]$. The ideas of the following lemma can be found implicitly in \cite{taubes1996counting}.

\begin{Lemma}\label{lem: regularity of moduli space for the compact case}
Let $X$ be a compact 4-manifold and let $\omega_s$ $(s\in[0,1])$ be a smooth family of symplectic forms on $X$. Let $e\in H_2(X;\mathbb{Z})$ be a primitive class with $\langle c_1(T^{0,1}X), e \rangle =0$ and $e\cdot e=0$, and let $E>0$ be a constant such that $E>\langle [\omega_s], e \rangle$ for every $s$. For $i\in\{0,1\}$, let $J_i\in \mathcal{J}(X,\omega_i)$ be an $E$-admissible almost complex structure on $X$. Then there is an open and dense subset $\mathcal{U}\subset\mathcal{J}(X,\{\omega_s\},J_0,J_1)$ in the $C^{\infty}$-topology, such that for every element $\{J_s\}\in\mathcal{U}$, the moduli space $\mathcal{M}(X,\{J_s\},e)=\coprod_{s\in [0,1]}\mathcal{M}(X,J_s,e)$ has the structure of a compact smooth 1-manifold with boundary $\mathcal{M}(X,J_0,e) \cup \mathcal{M}(X,J_1,e)$.
\end{Lemma}

\begin{proof}
The formal dimension of the moduli space of $J_s$-holomorphic maps from a genus $g$ curve to $X$ in homology class $e$, modulo self-isomorphisms of the domain, is equal to $2(g-\langle c_1(T^{0,1}X),e \rangle -1)$, which always even. When the formal dimension is negative, it is less than or equal to $-2$. Therefore, there is an open and dense subset $\mathcal{U_1}\subset\mathcal{J}(X,\{\omega_s\},J_0,J_1)$, such that condition \ref{condition 2} of definition \ref{def: admissible} holds for each $J_s$. The standard transversality argument shows that on an open and dense subset $\mathcal{U}\subset \mathcal{U}_0$, the space $\mathcal{M}(X,\{J_s\},e)$ is a smooth 1-manifold with boundary $\mathcal{M}(X,J_0,e) \cup \mathcal{M}(X,J_1,e)$. For general $X$ and $e$ the spcae $\mathcal{M}(X,\{J_s\},e)$ does not have to be compact. However, since it is assumed that $\pi_2(X)=0$, there is no non-constant $J_s$-holomorphic maps from a sphere to $X$. By Gromov's compactness theorem (see for example \cite{ye1994gromov}), for every sequence $\{C_n\}\subset \mathcal{M}(X,\{J_s\},e)$, there is a subsequence $\{C_{n_i}\}$ with $C_{n_i}\in \mathcal{M}(X,J_{s_i},e)$ and $\lim_{i\to\infty}s_i=s_0$, such that the sequence $C_{n_i}$ is convergent to one of the following:  (1) a branched multiple cover of a somewhere injective $J_{s_0}$-holomorphic map, (2) a somewhere injective $J_{s_0}$-holomorphic map with bubbles or nodal singularities or both, (3) a somewhere injective $J_{s_0}$-holomorphic torus. Case (1) is impossible since $e$ is assumed to be a primitive class. Case (2) is impossible becase there is no non-constant $J_{s_0}$-holomorphic maps from a sphere to $X$.  When case (3) happens, for the limit curve the adjunction formula states that $e\cdot e + 2-2g = -\langle c_1(T^{0,1}X), e \rangle + \kappa$, where $\kappa$ depends on the behaviour of singularities and self-intersections of the curve, and $\kappa$ is always positive if the curve is not embedded (see \cite{mcduff1991the}). Since $g=1$, $e\cdot e=0$, $\langle c_1(T^{0,1}X), e \rangle =0$, it follows that $\kappa=0$, hence the limit curve is an embedded curve, namely it is an element of $\mathcal{M}(X,J_{s_0},e)$. Therefore the space $\mathcal{M}(X,\{J_s\},e)$ is compact.
\end{proof}

With a little more effort one can generalize lemma \ref{lem: regularity of moduli space for the compact case} to non-compact symplectic manifolds. To start, one needs the following definition.

\begin{Definition}
Let $(X,\omega)$ be a symplectic manifold, not necessarily compact. Let $J\in\mathcal{J}(X,\omega)$. The pair $(\omega,J)$ defines a Riemannian metric $g$ on $X$. The triple $(X,\omega,J)$ is said to have bounded geometry with bounding constant $N$ if the following conditions hold:
\begin{enumerate}
\item The metric $g$ is complete.
\item The norm of the curvature tensor of $g$ is less than $N$.
\item The injectivity radius of $(X,g)$ is greater than $1/N$.
\end{enumerate}
\end{Definition}

One says that a path $\{(X,\omega_s,J_s)\}$ has uniformly bounded geometry if each $(X,\omega_s,J_s)$ has bounded geometry, and the bounding constant $N$ is independent of $s$.

The following lemma is a well-known result.

\begin{Lemma} \label{lem: bound on diameters of J curve}
Let $(X,\omega,J)$ be a triple with bounded geometry, with bounding constant $N$. Let $e\in H_2(X;\mathbb{Z})$, and let $E>0$ be a constant such that $E\ge \langle [\omega], e\rangle$. Then there is a constant $M(N,E)$, depending only on $N$ and $E$, such that every connected $J$-holomorphic curve $C$ with fundamental class $e$ has diameter less than $M(N,E)$ with respect to the metric defined by $\omega(\cdot,J\cdot)$.
\end{Lemma}

\begin{proof}
By the monotonicity of area, there is a constant $\delta$ depending only on $N$, such that for every point $p\in C$ the area of $B_p(1/N)\cap C$ is greater than $\delta$. Since $C$ is connected, this implies that the total area of $C$ bounds its diameter. Notice that the area of $C$ equals $\langle [\omega], e\rangle $, which is bounded by $E$, hence the the diameter is bounded by a function of $N$ and $E$.
\end{proof}

In the noncompact case, one needs to be more careful about the topology of the spaces of almost complex structures. A topology on $\mathcal{J}(X,\omega)$ can be defined as follows.
Cover $X$ by countably many compact sets $\{A_i\}_{i\in\mathbb{Z}}$. For each $A_i$ define the $C^\infty$-topology on $\mathcal{J}(A_i,\omega)$. Endow the product space
$$
\prod_{i\in\mathbb{Z}}\mathcal{J}(A_i,\omega)
$$
with the box topology, and consider the map
$$
\mathcal{J}(X,\omega) \xhookrightarrow{} \prod_{i\in\mathbb{Z}}\mathcal{J}(A_i,\omega)
$$
defined by restrictions. The topology on $\mathcal{J}(X,\omega)$ is then defined as the pull back of the box topology on the product space.

For $N>0$, define $\mathcal{J}(X,\omega,N)$ to be the set of almost complex structures $J\in\mathcal{J}(X,\omega)$ such that $(X, \omega, J)$ has bounded geometry with bounding constant $N$. With the topology given above, the space $\mathcal{J}(X,\omega,N)$ is an open subset of $\mathcal{J}(X,\omega)$.

A topology on $\mathcal{J}(X,\{\omega_s\},J_0,J_1)$ can be defined in a similar way. Cover $X$ by countably many compact sets $\{A_i\}_{i\in\mathbb{Z}}$. For each $A_i$ define the $C^\infty$-topology on $\mathcal{J}(A_i,\{\omega_s\},J_0,J_1)$. The topology on the space $\mathcal{J}(X,\{\omega_s\},J_0,J_1)$ is then defined as the pull back of the box topology on the product space.

For $N>0$, define the set $\mathcal{J}(X,\{\omega_s\},J_0,J_1,N)$ to be the set of families $\{J_s\}\in \mathcal{J}(X,\{\omega_s\},J_0,J_1)$ such that $\{(X,J_s,\omega_s)\}$ has uniformly bounded geometry with bounding constant $N$. Then the set $\mathcal{J}(X,\{\omega_s\},J_0,J_1,N)$ is an open subset of $\mathcal{J}(X,\{\omega_s\},J_0,J_1)$.

The following lemma is essentially a diagonal argument. It explains why the topologies defined above are the correct topologies to accommodate the perturbation arguments for the rest of the article.

\begin{Lemma} \label{lem: diagonal argument}
Let $\{A_n\}_{n\ge 1}$ be a countable, locally finite cover of $X$ by compact subsets. Let $\omega$ be a symplectic form on $X$, let $\omega_s$ be a smooth family of symplectic forms on $X$. Let $N>0$ be a constant. Let $J_i\in \mathcal{J}(X,\omega_i,N)$, where $i=0\mbox{ or }1$.
\begin{enumerate}
\item Let $\varphi:\mathcal{J}(X,\omega) \xhookrightarrow{} \prod_n\mathcal{J}(A_n,\omega)$ be the embedding map. For every $n$, let $\mathcal{U}_n$ be an open and dense subset of $\mathcal{J}(A_n,\omega)$, then $\varphi^{-1}(\prod_n \mathcal{U}_n)$
is an open and dense subset of $\mathcal{J}(X,\omega)$.
\item Let $\varphi:\mathcal{J}(X,\{\omega_s\},J_0,J_1) \xhookrightarrow{} \prod_n\mathcal{J}(A_i,\{\omega_s\},J_0,J_1)$ be the embedding map. For every $n$, let $$\mathcal{U}_n\subset \mathcal{J}(A_n,\{\omega_s\},J_0,J_1) $$ be an open and dense subset, then $\varphi^{-1}(\prod_n \mathcal{U}_n)$
is an open and dense subset of $\mathcal{J}(X,\{\omega_s\},J_0,J_1) $.
\end{enumerate}
\end{Lemma}

\begin{proof}
For part 1, the set $\varphi^{-1}(\prod_n \mathcal{U}_n)$ is open by the definition of box topology. To prove that $\varphi^{-1}(\prod_n \mathcal{U}_n)$ is dense, let $J$ be an element of $\mathcal{J}(X,\omega)$. Let $J_n=J|_{A_n}\in\mathcal{J}(A_n,\omega)$. Let $\mathcal{V}_n\subset \mathcal{J}(A_n,\omega)$ be an arbitrary open neighborhood of $J_n$. One needs to find an element $J'\in \mathcal{J}(X,\omega)$ such that $J'|_{A_n}\in \mathcal{V}_n\cap \mathcal{U}_n$. For each $n$, let $D_n$ be an open neighborhood of $A_n$ such that the family $\{D_n\}$ is still a locally finite cover of $X$. One obtains the desired $J'$ by perturbing $J$ on the open sets $\{D_n\}$ one by one. To start, perturb the section $J$ on $D_1$ to obtain a section $J_1$. Since $\mathcal{U}_1$ is dense it is possible to find a perturbation such that $J_1|_{A_1}\in \mathcal{U}_1\cap \mathcal{V}_1$. Now assume that after perturbation on $D_1,D_2,\cdots, D_k$, one obtains a section $J_k$ such that $J_k|_{A_i}\in \mathcal{U}_j\cap \mathcal{V}_j$ for $j=1,2,\cdots,k$. Then a perturbation of $J_k$ on $D_{k+1}$ gives a section $J_{k+1}$ such that $J_{k+1}|_{A_{k+1}}\in \mathcal{U}_{k+1}\cap \mathcal{V}_{k+1}$. When the perturbation is small enough, it still has the property that $J_{k+1}|_{A_j}\in \mathcal{U}_j\cap \mathcal{V}_j$ for $j=1,2,\cdots,k$. Since $\{D_n\}$ is a locally finite cover of $X$, on each compact subset of $X$ the sequence $\{J_k\}$ stabilizes for sufficiently large $k$. The limit $\lim_{k\to \infty} J_k$ then gives the desired $J'$.

The proofs for part 2 is exactly the same, one only needs to change the notation $\mathcal{J}(\cdot,\omega)$ to $\mathcal{J}(\cdot,\{\omega_s\},J_0,J_1)$.
\end{proof}

\begin{Lemma} \label{lem: regularity of moduli space for the noncompact case}
Let $X$ be a 4-manifold, let $e\in H_2(X;\mathbb{Z})$ be a primitive class. Assume $\omega_s$ $(s\in[0,1])$ is a smooth family of symplectic forms on $X$. Let $E$ be a positive constant such that $E>\langle[\omega_s],e\rangle$ for every $s$. For $i=0,1$, assume $J_i\in \mathcal{J}(X,\omega_i,N)$ is $E$-admissible. If the set $\mathcal{J}(X,\{\omega_s\},J_0,J_1,N)$ is not empty, then there is an open and dense subset $\mathcal{U}\subset\mathcal{J}(X,\{\omega_s\},J_0,J_1,N)$, such that for each $\{J_s\}\in\mathcal{U}$, the moduli space $\mathcal{M}(X,\{J_s\},e)=\coprod_{s\in [0,1]}\mathcal{M}(X,J_s,e)$ has the structure of a smooth 1-manifold with boundary $\mathcal{M}(X,J_0,e) \cup \mathcal{M}(X,J_1,e)$. Moreover, if $f:X\to \mathbb{R}$ is a smooth proper function on $X$, then the function defined as
\begin{align*}
\mathfrak{f}:\mathcal{M}(X,\{J_s\},e) & \to \mathbb{R} \\
                   C   & \mapsto \big(\int_C f \, dA \big)/
                         \big(\int_C 1 \, dA \big)
\end{align*}
is a smooth proper function on $\mathcal{M}(X,\{J_s\},e)$, where $dA$ is the area form of $C$.
\end{Lemma}

\begin{proof}
One first prove that there is an open and dense subset $\mathcal{U}\subset\mathcal{J}(X,\{\omega_s\},J_0,J_1,N)$, such that for every $\{J_s\}\in\mathcal{U}$, the moduli space $\mathcal{M}(X,\{J_s\},e)$ is a smooth 1-dimensional manifold. Let $g_s$ be the metric on $X$ compatible with $J_s$ and $\omega_s$. Let $g$ be a complete metric on $X$ such that $g_s\ge g$ for every $s$. From now on, the distance function on $X$ is defined by the metric $g$. By lemma \ref{lem: bound on diameters of J curve}, there exists a constant $M>0$ such that the diameter of every $J_s$-holomorphic curve with energy no greater than $E$ is bounded by $M$. Let $\{B_n\}$ be a countable locally finite cover of $X$ by open balls of radius $1$. For every $n$, let $A_n$ be the closed ball with the same center as $B_n$ and with radius $(M+1)$. The family $\{A_n\}$ is also a locally finite cover of $X$. For each $n$, let $\mathcal{M}_n(X,\{J_s\},e)$ be the open subset of $\mathcal{M}(X,\{J_s\},e)$ consisting of the curves $C\in \mathcal{M}(X,\{J_s\},e)$ such that $C\cap A_n\neq \emptyset$. By the diameter bound of $J_s$-holomorphic curves and the results for the compact case, there is an open and dense subset $\mathcal{U}_n\subset \mathcal{J}(A_n,\{\omega_s\},J_0,J_1,N)$ such that if $\{J_s\}|_{A_n}\in \mathcal{U}_n$, then the set $\mathcal{M}_n(X,\{J_s\},e)$ is a smooth 1 dimensional manifold. It then follows from part 2 of lemma \ref{lem: diagonal argument} that there is an open and dense subset $\mathcal{U}\subset\mathcal{J}(X,\{\omega_s\},J_0,J_1,N)$ such that for every element $\{J_s\}\in\mathcal{U}$ the set $\mathcal{M}(X,\{J_s\},e)$ is a smooth 1-manifold.

When set $\mathcal{M}(X,\{J_s\},e)$ is a smooth 1-manifold, its boundary is $\mathcal{M}(X,J_0,e) \cup \mathcal{M}(X,J_1,e)$, and the function $\mathfrak{f}$ is a smooth function on $\mathcal{M}(X,\{J_s\},e)$. 

It remains to prove that $\mathfrak{f}$ is a proper function. For any constant $z>0$, take a sequence of curves $C_n\in \mathcal{M}(X,\{J_s\},e)$ such that $|\mathfrak{f}(C_n)|<z$. By the definition of $\mathfrak{f}$, there exists a sequence of points $p_n\in C_n$ such that $|f(p_n)|<z$. Since $f$ is a proper function on $X$, the sequence $p_n$ is bounded on $X$. By lemma \ref{lem: bound on diameters of J curve} this implies that the curves $C_n$ stay in a bounded subset of $X$. By the argument for the compact case (lemma \ref{lem: regularity of moduli space for the compact case}), the sequence $\{C_n\}$ has a subsequence that converges to another point in $\mathcal{M}(X,\{J_s\},e)$, hence function $\mathfrak{f}$ is proper.
\end{proof}

\section{Symplectization of taut foliations} \label{sec: symplectization}
This section discusses a symplectization of \emph{oriented and cooriented} taut foliations. It is the main ingredient for the proof of theorem \ref{thm: main}.

Let $M$ be a smooth 3-manifold, let $\mathcal{F}$ be a smooth oriented and cooriented taut foliation on $M$. Since $\mathcal{F}$ is cooriented, it can be written as $\mathcal{F}=\ker\lambda$ where the positive normal direction of $\mathcal{F}$ is positive on $\lambda$. 
Since $\mathcal{F}$ is taut, there exists a closed 2-form $\omega$ such that $\omega\wedge\lambda>0$ everywhere on $M$. Choose a metric $g_0$ on $M$ such that $*_{g_0}\lambda=\omega$. 
By Frobenious theorem, $d\lambda=\mu\wedge\lambda$ for a unique 1-form $\mu$ satisfying $\mu\perp\lambda$. 
Locally, write $\omega=e^1\wedge e^2$ where $e^1$ and $e^2$ are orthonormal with respect to the metric $g_0$. 
Consider the 2-form $\Omega=\omega+d(t\lambda)$ on $M\times \mathbb{R}$ and the metric $g$ defined by
$$
g=\frac{1}{1+t^2}\cdot(dt+t\mu)^2+(1+t^2)\lambda^2+(e^1)^2+(e^2)^2
$$

The 2-form $\Omega$ is a symplectic form on $M\times \mathbb{R}$, and the metric $g$ is independent of the choice of $\{e^1,e^2\}$ and is compatible with $\Omega$. Let $J$ be the almost complex structure given by $(\Omega,g)$. To simplify notations, let $X$ be the manifold $M\times \mathbb{R}$.
\begin{Lemma}[\cite{zhang2016monopole}, lemma 2.1] \label{lem:X has bounded geometry}
The triple $(X,\Omega,J)$ has bounded geometry. \qed
\end{Lemma}

Locally, let $\{e_0,e_1,e_2\}$ be the basis of $TM$ dual to $\{\lambda,e^1,e^2\}$, and extend them to $\mathbb{R}$-translation invariant vector fields on $M\times \mathbb{R}$. Let $\hat{e}_1=e_1-t\mu(e_1)\frac{\partial}{\partial t}$, $\hat{e}_2=e_2-t\mu(e_2)\frac{\partial}{\partial t}$. The almost complex structure $J$ is then given by 
\begin{align*}
J\frac{\partial}{\partial t} &=\frac{1}{1+t^2}e_0, \\
J \hat{e}_1 &= \hat{e}_2.
\end{align*}

Define $\widetilde{\mathcal{F}}= \textrm{span}\{\hat{e}_1,\hat{e}_2\}$, it is a $J$-invariant plane field on $X$.
\begin{Lemma}
The plane field $\widetilde{\mathcal{F}}$ is a foliation on $X$. Under the projection $M\times \mathbb{R} \to M$, the leaves of $\widetilde{\mathcal{F}}$ projects to the leaves of $\mathcal{F}$.
\end{Lemma}

\begin{proof}
Since $d\mu\wedge\lambda =d(d\lambda)=0$, there is a $\mu_1$ such that $d\mu=\mu_1\wedge \lambda$. 
Therefore, one has $d(dt+t\mu)=(dt+t\mu)\wedge \mu + t\mu_1\wedge\lambda$, and $d\lambda=\mu\wedge \lambda$. 
By Frobenius theorem, the plane field $\widetilde{\mathcal{F}}=\ker (dt+t\mu) \cap \ker \lambda$ is a foliation. 
The tangent planes of $\widetilde{\mathcal{F}}$ projects isomorphically to the tangent planes of $\mathcal{F}$ pointwise, thus the leaves of $\widetilde{\mathcal{F}}$ projects to the leaves of $\mathcal{F}$. 
\end{proof}

It turns out that every closed $J$-holomorphic curve in $X$ is a closed leaf of $\widetilde{\mathcal{F}}$.
\begin{Lemma} \label{lem:every J curve is a leaf}
Let $\rho:\Sigma\to X$ be a $J$-holomorphic map from a closed Riemann surface to $X$. Then either $\rho$ is a constant map, or it is a branched cover of a closed leaf of $\widetilde{\mathcal{F}}$.
\end{Lemma}

\begin{proof}
Since $\rho$ is $J$-holomorphic, $\rho^*\big((dt+t\mu)\wedge \lambda\big)\ge 0$ pointwise on $\Sigma$. On the other hand,
$$
\int_\Sigma \rho^*\big( (dt+t\mu)\wedge \lambda\big) = \int_\Sigma \rho^*(d(t\lambda)) =0.
$$
Therefore $\rho(\Sigma)$ is tangent to $\ker (dt+t\mu) \cap \ker \lambda$, hence either $\rho$ is a constant map, or it is a branched cover of a closed leaf of $\widetilde{\mathcal{F}}$.
\end{proof}

\begin{Lemma} \label{lem: holonomy of tilde F}
Let $L$ be a leaf of $\mathcal{F}$ and $\gamma$ a closed curve on $L$. Let $\pi:M\times \mathbb{R}\to M$ be the projection map. The foliation $\widetilde{\mathcal{F}}$ is then transverse to $\pi^{-1}(\gamma)$ and gives a horizontal foliation on $\pi^{-1}(\gamma)\cong \gamma\times \mathbb{R}$. The holonomy of this foliation along $\gamma$ is given by multiplication of $l(\gamma)^{-1}$, where $l(\gamma)$ is the linear holonomy of $\mathcal{F}$ along $\gamma$.
\end{Lemma}

\begin{proof} 
Suppose $\gamma$ is parametrized by $u\in[0,1]$. Let $(\gamma(u),t(u))$ be a curve in $M\times\mathbb{R}$ that is a lift of $\gamma$ and tangent to $\widetilde{\mathcal{F}}$. Then the function $t(u)$ satisfies $\dot{t}+t\mu(\dot{\gamma})=0$. Therefore 
$$
t(1)=e^{-\int_0^1 \mu(\dot{\gamma}) du} \,t(0).
$$

Now let $U$ be a tubular neighborhood of $\gamma$ on the leaf $L$, and let $U\times (-\epsilon,\epsilon)\subset M$ be a tubular neighborhood of $U$ in $M$. 
Parametrize the second factor of $U\times (-\epsilon,\epsilon)$ by $z$, then on this neighborhood of $\gamma$ the 1-form $\lambda$ can be written as $f\cdot dz+\nu(z)$, where $f$ is a nowhere zero function on $U\times(-\epsilon, \epsilon)$, and $\nu(z)$ is a 1-form on $U$ depending on $z$ with $\nu(0)=0$. The restriction of the 1-form $\mu$ on $U$ then has the form $\mu=f^{-1}\,df-f^{-1}\,\frac{\partial\nu}{\partial z}|_{z=0}+g\cdot dz$ for some function $g$. If $(\gamma(u),z(u))$ is a curve in $U\times (-\epsilon,\epsilon)$ tangent to $\mathcal{F}$, then 
\begin{equation} \label{eqn: holonomy}
\dot{z}+f(\gamma,z)^{-1}\cdot\nu(z)(\dot{\gamma})=0.
\end{equation}
If $z_s(u)$, $s\in[0,\epsilon)$ is a smooth family of solutions to \eqref{eqn: holonomy} with $z_0(u)=0$, then the linearized part $l(u)=\frac{\partial z_s}{\partial s}|_{s=0}(u)$ satisfies
$$
\dot{l}+l\cdot f^{-1}(\gamma,0)\cdot\frac{\partial\nu}{\partial z}\Big|_{z=0}(\dot{\gamma})=0.
$$
Therefore the linear holonomy of $\mathcal{F}$ along $\gamma$ is 
$$
e^{-\int_0^1 f^{-1}\cdot\frac{\partial\nu}{\partial z}(0)(\dot{\gamma}(u)) du },$$
which is equal to $e^{\int_0^1 \mu(\dot{\gamma}) du}$, hence the linear holonomy of $\mathcal{F}$ along $\gamma$ is inverse to the holonomy on $\pi^{-1}(\gamma)$ given by $\widetilde{\mathcal{F}}$.
\end{proof}

The following result follows immediately from lemmas \ref{lem:every J curve is a leaf} and \ref{lem: holonomy of tilde F}.

\begin{Corollary} \label{cor:closed leaf of tilde F}
Let $C$ be a closed embedded $J$-holomorphic curve on $X$. Then either $C\subset M\times\{0\}$ and $C$ is a closed leaf of $\mathcal{F}$, or $C$ does not intersect the slice $M\times\{0\}$ and it projects diffeomorphically to a closed leaf of $\mathcal{F}$ with trivial linear holonomy.  \qed
\end{Corollary}

The next lemma studies $J$-holomorphic tori on $X$.

\begin{Lemma}\label{lem: nondegeneracy}
Suppose $T$ is a torus leaf of $\mathcal{F}$ with non-trivial linear holonomy. Then $T\times \{0\}$ is a nondegenerate $J$-holomorphic curve in $X$.
\end{Lemma}

\begin{proof}
Notice that $d([T])=0$, thus the index of the deformation operator is zero, and one only needs to prove that for $T$ the operator $L$ defined by equation \eqref{eqn: deformation operator} has a trivial kernel.

Let $T_0=T\times \{0\}$ be the torus in $X$. Recall that locally $(\lambda,e^1,e^2)$ is an orthonormal basis of $T^*M$ and $(e_0,e_1,e_2)$ is its dual basis. Let $T\times (-\epsilon,\epsilon)\subset M$ be a tubular neighborhood of $T$ in $M$ such that the fibers of $(-\epsilon,\epsilon)$ are flow lines of $e_0$.
Choose a parametrization $z$ for the second factor of $T\times (-\epsilon,\epsilon)$, such that $\lambda(\frac{\partial}{\partial z})=1$. Then on this neighborhood $e_0=\frac{\partial}{\partial z}$, and $\lambda$ has the form $\lambda=dz+\nu(z)$ where $\nu(z)$ is a 1-form on $T$ depending on $z$ and $\nu(0)=0$. The condition that $\ker \lambda$ is a foliation is equivalent to 
$$
d\nu+\frac{\partial\nu}{\partial z}\wedge\nu=0.
$$
Let $\beta=\frac{\partial\nu}{\partial z}|_{z=0}$. Apply $\frac{\partial}{\partial z}$ on the equation above at $z=0$, one obtains $d\beta=0$. Let $\lambda'=dz+z\cdot \beta$, then $\ker\lambda'$ defines another foliation near $T$. Let $\mu'= -\beta$.

Let $e_1'$, $e_2'$ be vector fields on $T\times(-\epsilon,\epsilon)$ such that they are tangent to $\ker \lambda'$, and their projections to $T$ form a positive orthonormal basis. Extend $e_1'$, $e_2'$ to a neighborhood of $T_0$ in $X$ by translation on the $t$-coordinate. Define an almost complex structure $J'$ on $T\times (-\epsilon,\epsilon)$ as
\begin{align*}
J'\frac{\partial}{\partial t} &=\frac{\partial}{\partial z}, \\
J' (e_1'-t\mu'(e_1')\frac{\partial}{\partial t}) &= e_2'-t\mu'(e_2')\frac{\partial}{\partial t} .
\end{align*}

Since $T$ has nontrivial linear holonomy, the same argument as in lemma \ref{lem:every J curve is a leaf} and lemma \ref{lem: holonomy of tilde F} shows that $T_0$ is the only embedded $J'$-holomorphic torus in a neighborhood of $T_0$. On the other hand, a straight forward calculation shows that the equation \eqref{eqn:J holomorphic for sections} for deformation of $J'$-holomorphic curves near $T_0$ is a linear equation, therefore $T_0$ is nondegenerate as a $J'$-holomorphic curve. Since $J'$ and $J$ agree up to first order derivatives along the curve $T_0$, this proves that $T_0$ is nondegenerate with respect to $J$.
\end{proof}

\section{Proof of theorem \ref{thm: main}} \label{sec: proof}

Now let $\mathcal{L}$ be a cooriented smooth taut foliation on a smooth 3-manifold $Y$. Consider its orientation double cover $\widetilde{\mathcal{L}}$. It is an oriented and cooriented taut foliation on the orientation double cover of $Y$. Let $p:\widetilde{Y} \to Y$ be the covering map. If $K$ is a Klein-bottle leaf of $\mathcal{L}$, then $p^{-1}(K)$ is a torus leaf of $\widetilde{L}$. Recall that in the beginning of section \ref{sec: statement of result}, a homology class $PD[K]\in H^1(Y;\mathbb{Z})$ was defined for every Klein-bottle leaf.

\begin{Lemma} \label{lem: pull back of Poincare dual}
Let $K$ be a Klein-bottle leaf of $\mathcal{L}$. Let $PD[p^{-1}(K)]$ be the Poincare dual of the fundamental class of $p^{-1}(K)$. Then $p^*(PD[K])=PD[p^{-1}(K)]$.
\end{Lemma}
\begin{proof}
Let $\gamma$ be a closed curve in $\widetilde{Y}$. Use $I(\cdot,\cdot)$ to denote the intersection number. Then
\begin{align*}
\langle PD[p^{-1}(K)], [\gamma] \rangle & = I(p^{-1}(K),\gamma) \\
&=I(K,p(\gamma)) = \langle PD[K], p_*[\gamma] \rangle =\langle p^*(PD[K]), [\gamma] \rangle.
\end{align*}
Therefore $p^*(PD[K])=PD[p^{-1}(K)]$.
\end{proof}

\begin{Lemma}\label{lem:pull back is injective}
The pull-back map $p^*:H^1(Y;\mathbb{Z})\to H^1(\widetilde{Y};\mathbb{Z})$ is injective.
\end{Lemma}

\begin{proof} 
Every element in $\ker p^*$ is represented by an element  $\alpha\in\Hom ( \pi_1(S), \mathbb{Z} )$ such that $\alpha$ is zero on the image of $p_*:\pi_1(\widetilde{Y})\to \pi_1(Y)$. Since $\Ima p_*$ is a normal subgroup of $\pi_1(Y)$ of index 2, the map $\alpha$ is decomposed as 
$$
\alpha: \pi_1(Y) \to \pi_1(Y)/\pi_1(\widetilde{Y}) \cong \mathbb{Z}/2 \to \mathbb{Z},
$$
which has to be zero. Therefore $p^*$ is injective.
\end{proof}

By lemma \ref{lem: pull back of Poincare dual} and \ref{lem:pull back is injective}, a Klein-bottle leaf $K$ has $PD[K]=A$ if and only if $PD([p^{-1}(K)])=p^*(A)$.
The next lemma shows that for every Klein-bottle leaf $K$ of $\mathcal{L}$ the fundamental class $[p^{-1}(K)]$ is a primitive class.

\begin{Lemma}\label{lem: primitive}
Let $\mathcal{F}$ be an oriented and cooriented taut foliation on a smooth three manifold $M$, then the fundamental class of every closed leaf of $\mathcal{F}$ is a primitive class.
\end{Lemma}
 
\begin{proof}
Let $L$ be a closed leaf of $\mathcal{F}$. Take a point $p\in L$. By the definition of tautness, there exists an embedded circle $\gamma$ passing through $p$ and transverse to the foliation. Let $\gamma:[0,1]\to M$ with $\gamma(0)=\gamma(1)=p$ be a parametrization of $\gamma$. By transversality, $\gamma^{-1}(L)$ is a finite set. Let $t_0$ be the minimum value of $t>0$ such that $\gamma(t_0)\in L$. Then for $\epsilon$ sufficiently small one can slide the part of $\gamma$ on $(t_0-\epsilon, t_0+\epsilon)$ along the foliation, such that the resulting curve is still transverse to $\mathcal{F}$, and such that $\gamma(t_0)=p$. Now $\gamma|_{[0,t_0]}$ defines a circle whose intersection number with $L$ equals 1. The existence of such a curve implies that the fundamental class of $L$ is primitive.
\end{proof}

With the preparations above, one can now prove theorem \ref{thm: main}.

\begin{proof}[Proof of theorem \ref{thm: main}]
Let $A\in H^1(Y;\mathbb{Z})$. Suppose $\mathcal{L}_0$ and $\mathcal{L}_1$ are two smooth $A$-admissible taut foliations on $Y$, such that they can be deformed to each other by a smooth family of taut foliations $\mathcal{L}_s$, $s\in [0,1]$. Let $\widetilde{Y}$ be the orientation double cover of $Y$. Then the orientation double covers $\widetilde{\mathcal{L}}_s$ pf $\mathcal{L}_s$ form a smooth family of oriented and cooriented taut foliaitons on $\widetilde{Y}$. 

Let $\tilde{\sigma}: \widetilde{Y} \to \widetilde{Y}$ be the deck transformation of the orientation double cover. Then the map $\tilde{\sigma}$ preserves the coorientation of $\widetilde{L}_s$ and reverses its orientation for each $s$. 

There exists a smooth family of 1-forms $\lambda_s$ and closed 2-forms $\omega_s$ on $\widetilde{Y}$ such that $\widetilde{L}_s=\ker\lambda_s$ and $\lambda_s\wedge\omega_s>0$. By changing $\lambda_s$ to $(\lambda_s+\tilde{\sigma}^*\lambda_s)/2$ and changing $\omega_s$ to $(\omega_s-\tilde{\sigma}^*\omega_s)/2$, one can assume that $\tilde{\sigma}^*\lambda_s=\lambda_s$, and $\tilde{\sigma}^*\omega_s=-\omega_s$. Let $(\Omega_s,J_s)$ be the corresponding symplectic structures and almost complex structures on $X=\widetilde{Y}\times\mathbb{R}$. Define 
\begin{eqnarray*}
 \sigma:  &           X  &\to         X \\
          &         (x,t) &\mapsto (\tilde{\sigma}(x),-t).
\end{eqnarray*}
Then $\sigma^*(\Omega_s)=-\Omega_s$, and $\sigma^*(J_s)=-J_s$. The family $\{(X,\Omega_s,J_s)\}$ has uniformly bounded geometry. This means that there is a constant $N>0$ such that $J_s\in \mathcal{J}(X,\Omega_s,N)$ for each $s$. 

If neither $\mathcal{L}_0$ nor $\mathcal{L}_1$ has any Klein-bottle leaf in the class $A$, the statement of theorem \ref{thm: main} obviously holds. From now on assume that either $\mathcal{L}_0$ or $\mathcal{L}_1$ has at least one Klein-bottle leaf in the class $A$. Let $e$ be the push forward of $PD(p^*(A))\in H_2(\widetilde{Y};\mathbb{Z})$ to $H_2(X;\mathbb{Z})$ via the inclusion map $\widetilde{Y}\cong \widetilde{Y}\times\{0\} \xhookrightarrow{} X$. The class $e$ then satisfies $\sigma_*(e)=-e$. By lemma \ref{lem: primitive}, $e$ is a primitive class. Roussarie-Thurston theorem implies that $\pi_2(X)=\pi_2(\widetilde{Y}) = 0$. 

Take a positive constant $E$ such that $E>\langle [\Omega_s], e \rangle $ for all $s$. Let $M(N,E)$ be the diameter upper bound from lemma \ref{lem: bound on diameters of J curve} for the geometry bound $N$ and the energy bound $E$. Let $T_0>0$ be sufficiently large such that the distance of $\widetilde{Y}\times \{T_0\}$ and $\widetilde{Y}\times \{-T_0\}$ is greater than $M(N,E)+1$ for every metric $g_s$ induced from $(\Omega_s,J_s)$. 

For $i=0,1$, the union of torus leaves $L$ in $\widetilde{\mathcal{L}}_i$ in the homology class $p^*(A)$ such that $\int_L\omega_i\le E$ and $L$ is not the lift of any Klein-bottle leaf form a compact set $\widetilde{B}_i$. The set $\widetilde{B}_i$ satisfies $\tilde{\sigma}(\widetilde{B}_i)=\widetilde{B}_i$. Let $\widetilde{U}_i$ be a neighborhood of $\widetilde{B}_i$ such that $\tilde{\sigma}(\widetilde{U}_i)=\widetilde{U}_i$ and the closure of $\widetilde{U}_i$ does not intersect the lift of any Klein-bottle leaf of $\mathcal{L}_i$. Let 
$$
V=\widetilde{Y}\times \big((-\infty, -T_0) \cup  (T_0,\infty)\big),
$$
$$
U_i=\big(\widetilde{U}_i\times \mathbb{R} \big)\bigcup\widetilde{Y}\times \big((-\infty, -T_0) \cup  (T_0,\infty)\big)
$$
which are open subsets of $X$.
The following two lemmas will be proved in section \ref{sec: technical lemmas}.
\begin{Lemma} \label{lem: Z2 admissibility of Ji}
The almost complex structure $J_i$ can be perturbed to $J_i'\in \mathcal{J}(X,\Omega_i,N)$, such that $J_i'=J_i$ near Klein-bottle leaves, and $J_i'$ is $E$-admissible. Moreover, $\sigma^*(J_i')=-J_i'$ on $U_i$, and every $J_i'$-holomorphic torus of $X$ in the homology class $e$ is either contained in $U_i$ or is the lift of a Klein-bottle leaf in $\mathcal{L}_i$ in the class $A$. If $C$ is a $J_i'$-holomorphic curve in the homology class $e$ contained in $U_i$, then $\sigma(C)\neq C$.
\end{Lemma}

\begin{Lemma} \label{lem: Z2 admissibility of a family for taut foliation}
The almost complex structures $J_0'$ and $J_1'$ given by lemma \ref{lem: Z2 admissibility of Ji} can be connected by a smooth family of almost complex structures $$J_s'\in \mathcal{J}(X, \Omega_s, N),$$ such that $\sigma^*(J_s')=-J_s'$ on $V$, and the moduli space $\mathcal{M}(X,\{J_s'\},e)=\coprod_{s\in [0,1]}\mathcal{M}(X,J_s',e)$ has the structure of a smooth 1-manifold with boundary $\mathcal{M}(X,J_0',e) \cup \mathcal{M}(X,J_1',e)$. Moreover, let $t:X\to \mathbb{R}$ be the projection of $X=\widetilde{Y}\times\mathbb{R}$ to $\mathbb{R}$, then the function defined as
\begin{align*}
\mathfrak{f}:\mathcal{M}(X,\{J_s'\},e) & \to \mathbb{R} \\
                     C   & \mapsto \big(\int_C t \, dA\big)/\big(\int_C 1 \, dA\big).
\end{align*}
is a smooth proper function on $\mathcal{M}(X,\{J_s'\},e)$, where $dA$ is the area form of $C$.
\end{Lemma}

Let $\{J_s'\}$ be the family of almost complex structures given by the lemmas above, let $t_0>0$ be sufficiently large such that every $J_s'$-holomorphic torus $C$ in the homology class $e$ with $|\mathfrak{f}(C)|>t_0$ is contained in $V$. Take a constant $t_1>t_0$ such that $t_1$ and $-t_1$ are regular values of $\mathfrak{f}$, and that $t_1\notin\mathfrak{f}\big(\mathcal{M}(X,J_0',e) \cup \mathcal{M}(X,J_1',e)\big)$. Let $S_i=\mathcal{M}(X,J_i',e)\cap \mathfrak{f}^{-1}([-t_1,t_1])$. The set $\mathfrak{f}^{-1}(t_1)\cup\mathfrak{f}^{-1}(-t_1)\cup S_0\cup S_1$ is the boundary of the compact 1-manifold $\mathfrak{f}^{-1}([-t_1,t_1])$, hence it has an even number of elements. On the other hand, the properties of $\{J_s'\}$ given by lemma \ref{lem: Z2 admissibility of a family for taut foliation} shows that $\sigma$ maps $\mathfrak{f}^{-1}(t_1)$ to $\mathfrak{f}^{-1}(-t_1)$, therefore the set $\mathfrak{f}^{-1}(t_1)\cup\mathfrak{f}^{-1}(-t_1)$ has an even number of elements. The properties given by lemma \ref{lem: Z2 admissibility of Ji} implies that $\sigma$ acts on the set $S_i$, and the fixed point set of this action consists of the $J_i'$-holomorphic tori in $\widetilde{Y}\times \{0\}$ which are lifts of Klein-bottle leaves. Let $\mathcal{K}_i$ be the set of lifts of Klein-bottle leaves in $\mathcal{L}_i$ in the class $A$, then the arguments above shows that the number of elements in $\mathfrak{f}^{-1}(t_1)\cup\mathfrak{f}^{-1}(-t_1)\cup S_0\cup S_1$ has the same parity as the number of elements in $\mathcal{K}_0\cup \mathcal{K}_1$. Therefore, the set $\mathcal{K}_0\cup \mathcal{K}_1$ has an even number of elements, and the desired result is proved.
\end{proof}

\section{Technical lemmas} \label{sec: technical lemmas}

The purpose of this section is to prove lemma \ref{lem: Z2 admissibility of Ji} and lemma \ref{lem: Z2 admissibility of a family for taut foliation}. The proofs are routine and straightforward, they are given here for lack of a direct reference. Throughout this section $X$ will be a smooth 4-manifold with $\pi_2(X)=0$.

\begin{Definition}
Let $(X,\omega)$ be a symplectic manifold. Let $B\subset X$ be a closed subset. Let $E,N>0$ be constants. An almost complex structure $J\in \mathcal{J}(X,\omega,N)$ is called $(B,E)$-admissible if the following conditions hold:
\begin{enumerate}
\item Every embedded curve $C$ with energy less than or equal to $E$ and $d([C])=0$, and satisfies $C\cap B\neq\emptyset$ is nondegenerate.
\item For every homology class $e\in H_2(X;\mathbb{Z})$, if $\langle [\omega], e\rangle \le E$, and if $\langle c_1(T^{0,1}X), e\rangle >0$ (namely, the formal dimension of the moduli space of $J$-holomorphic maps from a torus to $X$ in the homology class $e$, modulo self-isomorphisms of the domain, is negative), then there is no somewhere injective $J$-holomorphic map $\rho$ from a torus to $X$ in the homology class $e$ such that $\textrm{Im}(\rho)\cap B\neq \emptyset$.
\end{enumerate}
\end{Definition}

The next lemma follows immediately from Gromov's compactness theorem and the diameter bound of lemma \ref{lem: bound on diameters of J curve}.
\begin{Lemma}
Let $(X,\omega)$ be a symplectic manifold. Let $B\subset X$ be a closed subset, and $E,N>0$ be constants. The elements of $\mathcal{J}(X,\omega,N)$ that are $(B,E)$-admissible form an open subset of $\mathcal{J}(X,\omega,N)$.\qed
\end{Lemma}

From now on assume that $\sigma:X\to X$ is a map that acts diffeomorphically on $X$, such that $\sigma^2=\textrm{id}_X$ and the quotient map $X\to X/\sigma$ is a covering map. 

\begin{Definition}
Let $(X,\omega)$ be a symplectic manifold. Let $d,E,N>0$ be constants. Let $B$ be a closed subset of $X$ such that $\sigma(B)=B$. An almost complex structure $J\in\mathcal{J}(X,\omega,N)$ is called $(d,E)$-regular with respect to $B$ if for every $J$-holomorphic map $\rho$ from a torus to $X$ with topological energy less than or equal to $E$, at least one of the following conditions hold:
\begin{enumerate}
\item The distance between the sets $\textrm{Im}(\rho)$ and $\sigma(\textrm{Im}(\rho))$ is greater than $d$. 
\item The distance of $\textrm{Im}(\rho)$ and $B$ is greater than $d$.
\end{enumerate}
Here the distance is defined by the metric $g_J=\omega(\cdot,J\cdot)$ on $X$.
\end{Definition}

Notice that since the map $\rho$ in the definition above can be a constant map, for a $(d,E)$-regular almost complex structure $J$ with respect to $B$, one has $\textrm{dist}(p,\sigma(p))>d$ for every $p\in B$. 

The following result is also a corollary of Gromov's compactness theorem.

\begin{Lemma} \label{lem: openness of regularity}
Let $d,E,N>0$ be constants, and $B$ is a closed subset of $X$ such that $\sigma(B)=B$. The elements of $\mathcal{J}(X,\omega,N)$ that are $(d,E)$-regular with respect to $B$ form an open subset of $\mathcal{J}(X,\omega,N)$.
\end{Lemma}

\begin{proof}
First consider the case when $B$ is compact. Let $M(N,E)$ be the upper bound of diameter given by lemma \ref{lem: bound on diameters of J curve}. Let $A$ be a compact set containing $B$ such that the distance between $\partial A$ and $B$ is greater than $M(N,E)+d+2$. Suppose $J$ is a $(d,E)$-regular almost complex structure with respect to $B$. Let $\mathcal{U}$ be a sufficiently small open neighborhood of $J|_A\in \mathcal{J}(A,\omega)$, such that for every $J'\in\mathcal{J}(X,\omega,N)$, if $J'|_A\in \mathcal{U}$ then the distance between $\partial A$ and $B$ is greater than $M(N,E)+d+1$. One claims that there is a smaller neighborhood $\mathcal{V}\subset \mathcal{U}$ containing $J$, such that for every $J'\in\mathcal{J}(X,\omega,N)$, if $J'|_A\in \mathcal{V}$ then $J'$ is $(d,E)$-regular with respect to $B$. In fact, assume the claim is not true, since $\mathcal{J}(A,\omega)$ is first countable, there is a sequence $\{J_n\}\subset \mathcal{J}(X,\omega,N)$, such that $J_n|_A\to J|_A$ in the $C^{\infty}$ topology, and that every $J_n$ is not $(d,E)$-regular with respect to $B$. By the definition of $(d,E)$-regularity, there is a sequence of $J_n$-holomorphic maps $\rho_n$ from torus to $X$ with topological energy less than or equal to $E$, such that the distance of $\textrm{Im}(\rho)$ to $B$ with respect to the metric given by $J_n$ is less than or equal to $d$, and the distance between $\textrm{Im}(\rho)$ and $\sigma(\textrm{Im}(\rho))$ with respect to the metric given by $J_n$ is less than or equal to $d$. By the diameter bound, every curve $C_n$ is contained in the set $A$. Gromov's compactness theorem then implies that there is a subsequence of $\rho_n$ such that at least part of the map converges to a non-constant $J$-holomorphic map. Since is it assumed that $\pi_2(X)=0$, the domain of the limit map is a torus. The limit map has topological energy less than or equal to $E$, and it violates the assumption that $J$ is $(d,E)$-regular with respect to $B$.

Now consider the case when $B$ is not necessarily compact. Let $J$ be a $(d,E)$-regular almost complex structure with respect to $B$. Cover $B$ by a locally finite family of compact subsets $B_n$ such that $\sigma(B_n)=B_n$ for each $n$. Let $A_n$ be the closed $(M(N,E)+d+2)$-neighborhood of $B_n$. By the argument of the previous paragraph, for each $n$ there is an open neighborhood $\mathcal{V}_n$ of $J|_{A_n}$ in $\mathcal{J}(A_n,\omega)$, such that for every $J'\in\mathcal{J}(X,\omega,N)$, if $J'|_{A_n}\in \mathcal{V}_n$ then $J'$ is $(d,E)$-regular with respect to $B_n$. Notice that $J'$ is $(d,E)$-regular with respect to $B$ if and only if it is $(d,E)$-regular with respect to every $B_n$. The result of the lemma then follows from part 1 of lemma \ref{lem: diagonal argument}.
\end{proof}

The following lemma is a 1-parametrized version of lemma \ref{lem: openness of regularity}.
\begin{Lemma} \label{lem: openness of regularity for family}
Let $d,E,N>0$ be constants, and $B$ is a closed subset of $X$ such that $\sigma(B)=B$. Let $\omega_s$ ($s\in[0,1]$) be a smooth family of symplectic forms on $X$, and let $J_i\in \mathcal{J}(X,\omega_i,N)$. Then the set of elements $\{J_s\}\in\mathcal{J}(X,\{\omega_s\},J_0,J_1,N)$ such that every $J_s$ is $(d,E)$-regular with respect to $B$ form an open subset of $\mathcal{J}(X,\{\omega_s\},J_0,J_1,N)$.
\end{Lemma}

\begin{proof}
The proof is exactly the same as lemma \ref{lem: openness of regularity}. One only needs to change the notation $J$ to $\{J_s\}$, and change the notation $\mathcal{J}(X,\omega,N)$ to $\mathcal{J}(X,\{\omega_s\},J_0,J_1,N)$.
\end{proof}

\begin{Lemma} \label{lem: Z2 admissibility}
Let $(X,\omega)$ be a symplectic manifold such that $\sigma^*(\omega)=-\omega$. Let $d,E,N>0$ be constants. Let $B$ be a closed subset of $X$ such that $\sigma(B)=B$. Assume $J\in\mathcal{J}(X,\omega,N)$ is $(d,E)$-regular with respect to $B$, and assume that $\sigma^*(J)=-J$ on $B$. Then for every open neighborhood $\mathcal{U}$ of $J$ in $\mathcal{J}(X,\omega,N)$, there is an element $J'$ such that $J'$ is $(d,E)$-regular with respect to $B$ and is $E$-admissible, and $\sigma^*(J')=-J'$ on $B$. Moreover, if there is a closed subset $H\subset X$ such that $\sigma(H)=H$ and $J$ is $(H,E)$-admissible, then $J'$ can be taken to be equal to $J$ on the set $H$.
\end{Lemma}

\begin{proof}
By shrinking the open neighborhood $\mathcal{U}$, one can assume that every element of $\mathcal{U}$ is $(d,E)$-regular with respect to $B$, and that there is a complete metric $g_0$ on $X$ such that $g_0\ge g_{J'}$ for every $J'\in\mathcal{U}$. For the rest of this proof, the distance function on $X$ is defined by $g_0$.

Cover $X$ by a locally finite family of closed balls with radius $d/10$. Say 
$$X = \bigcup_{i=1}^{N} B_i, $$
where $\{B_i\}$ are closed balls with radius $d/10$. Let $D_i$ be the open $d/10$-neighborhood of $B_i$.
Let $A_j=\cup_{i\le j} B_j$, where $A_0=\emptyset$. The construction of $J'$ follows from induction. Assume that $J_j$ is already $(A_j,E)$-admissible with $\sigma^*(J_j)=-J_j$ on $B$, the following paragraph will perturb $J_j$ to $J_{j+1}$ such that $J_{j+1}$ is $(A_{j+1},E)$-admissible with $\sigma^*(J_{j+1})=-J_{j+1}$ on $B$.

In fact, if $D_{j+1}\cap B=\emptyset$, then a generic perturbation on $D_{j+1}$ will do the job. If $D_{j+1}\cap B\neq\emptyset$, make a small perturbation on $D_{j+1}$ such that the resulting almost complex structure $J_{j+1}'$ is $(B_{j+1},E)$-admissible. Now make a corresponding perturbation on $\sigma(D_{j+1})$ such that the resulting almost complex structure $J_{j+1}$ satisfies $\sigma(J_{j+1})=-J_{j+1}$ on $B$. Since every element in $\mathcal{U}$ is $(d,E)$-regular with respect to $B$, there is no $J_{j+1}$-holomorphic map with topological energy less than or equal to $E$ and with image passing through both $D_{j+1}$ and $\sigma(D_{j+1})$, therefore $J_{j+1}'$ being $(D_{j+1},E)$-admissible implies that $J_{j+1}$ is $(D_{j+1},E)$-admissible. Since being $(A_j,E)$-admissible is an open condition, when the perturbation is sufficiently small the almost complex structure $J_{j+1}$ is also $(A_j,E)$-admissible. Therefore $J_{j+1}$ is $(A_{j+1},E)$-admissible. Since the family $\{D_n\}$ is locally finite, on each compact set the sequence $\{J_j\}$ stabilizes for sufficiently large $j$. The desired $J'$ can then be taken to be $\lim_{j \to \infty} J_j$. Moreover, if there is a closed subset $H\subset X$ such that $\sigma(H)=H$ and $J$ is $(H,E)$-admissible, then each step of the perturbation can be taken to be outside of $H$.\end{proof}

The following lemma is a 1-parametrized version of lemma \ref{lem: Z2 admissibility}, and the proof is essentially the same.

\begin{Lemma} \label{lem: Z2 admissibility for family}
Let $e\in H_2(X;\mathbb{Z})$ be a primitive class. Let $B$ be a closed subset of $X$ such that $\sigma(B)=B$. Assume $\omega_s$ $(s\in[0,1])$ is a smooth family of symplectic forms on $X$ such that $\sigma^*(\omega_s)=-\omega_s$ for each $s$. Let $d, N>0$ be constants. Let $E$ be a positive constant such that $E>\langle[\omega_s],e\rangle$ for every $s$. For $i=0,1$, assume $J_i\in \mathcal{J}(X,\omega_i,N)$ is $E$-admissible and $(d,E)$-regular with respect to $B$.  Assume $\{J_s\}\in\mathcal{J}(X,\{\omega_s\},J_0,J_1,N)$, such that for each $s$, the almost complex structure $J_s$ is $(d,E)$-regular with respect to $B$, and $\sigma^*(J_s)=-J_s$ on $B$. Then for every open neighborhood $\mathcal{U}$ of $\{J_s\}$ in $\mathcal{J}(X,\{\omega_s\},J_0,J_1,N)$, there is an element $\{J_s'\}$ such that $\{J_s'\}$ is $(d,E)$-regular with respect to $B$ and is $E$-admissible, and $\sigma^*(J_s')=-J_s'$ on $B$ for every $s$. Moreover, if there is a closed subset $H\subset X$ such that $\sigma(H)=H$ and $\{J_s\}$ is $(H,E)$-admissible, then $J_s'$ can be taken to be equal to $J_s$ on the set $H$.
\end{Lemma}

\begin{proof}
The proof follows verbatim as the proof of lemma \ref{lem: Z2 admissibility}. One only needs to change the notation $J$ to $\{J_s\}$, and change $\mathcal{J}(X,\omega,N)$ to $\mathcal{J}(X,\{\omega_s\},J_0,J_1,N)$.
\end{proof}

Combining the results above, one obtains the following lemma.
\begin{Lemma} \label{lem: regularity and admissibility for family}
Let $e\in H_2(X;\mathbb{Z})$ be a primitive class. Let $B$ be a closed subset of $X$ such that $\sigma(B)=B$. Assume $\omega_s$ $(s\in[0,1])$ is a smooth family of symplectic forms on $X$ such that $\sigma^*(\omega_s)=-\omega_s$ for each $s$. Let $d, N>0$ be constants. Let $E$ be a positive constant such that $E>\langle[\omega_s],e\rangle$ for every $s$. For $i=0,1$, assume $J_i\in \mathcal{J}(X,\omega_i,N)$ is $E$-admissible and $(d,E)$-regular with respect to $B$. Let $\mathcal{J}$ be the subset of elements $\{J_s\}$ of $\mathcal{J}(X,\{\omega_s\},J_0,J_1,N)$ such that for each $s$, the almost complex structure $J_s$ is $(d,E)$-regular with respect to $B$, and $\sigma^*(J_s)=-J_s$ on $B$. If $\mathcal{J}$ is not empty, let $\mathcal{U}\subset\mathcal{J}$ be the subset of $\mathcal{J}$, such that for every $\{J_s\}\in \mathcal{U}$, the moduli space $\mathcal{M}(X,\{J_s\},e)=\coprod_{s\in [0,1]}\mathcal{M}(X,J_s,e)$ has the structure of a smooth 1-manifold with boundary $\mathcal{M}(X,J_0,e) \cup \mathcal{M}(X,J_1,e)$. Then $\mathcal{U}$ is open and dense. Moreover, if $f:X\to \mathbb{R}$ is a smooth proper function on $X$, then the function defined as
\begin{align*}
\mathfrak{f}:\mathcal{M}(X,\{J_s\},e) & \to \mathbb{R} \\
                     C   & \mapsto \big(\int_C f \, dA\big)/\big(\int_C 1 \, dA\big).
\end{align*}
is a smooth proper function on $\mathcal{M}(X,\{J_s\},e)$, where $dA$ is the area form of $C$. \qed
\end{Lemma}

\begin{proof}
The openness of $\mathcal{U}$ follows from lemma \ref{lem: openness of regularity for family}. The fact that $\mathcal{U}$ is dense follows from lemma \ref{lem: Z2 admissibility for family}. The properness of the function $\mathfrak{f}$ was proved in lemma \ref{lem: regularity of moduli space for the noncompact case}.
\end{proof}

The following lemma controls the location of pseudo-holomorphic curves after perturbation of the almost complex structure.
\begin{Lemma} \label{lem: existence of J curve is a closed condition}
Let $(X,\omega)$ be a symplectic manifold, let $J\in \mathcal{J}(X,\omega,N)$. Let $E>0$ be a positive constant, and let $B$ be a closed subset of $X$.
Assume that there is no non-constant $J$-holomorphic map $\rho$ from a torus to $X$, such that $\textrm{Im}(\rho)\cap B$ is nonempty and the topological energy of $\rho$ is no greater than $E$. Then there is an open neighborhood $\mathcal{U}$ of $J$ in $\mathcal{J}(X,\omega,N)$, such that for every $J'\in \mathcal{U}$, there is no embedded $J'$-holomorphic torus in $X$ intersecting $B$ with energy less than or equal to $E$.
\end{Lemma}

\begin{proof}
Cover the set $B$ by a locally finite family of compact subsets $B_n$. Let $M(N,E)$ be the upper bound given by lemma \ref{lem: bound on diameters of J curve} for geometry bound $N$ and energy bound $E$. Let $A_n$ be the closed $M(N,E)+1$-neighborhood of $B_n$. 
One claims that there is an open neighborhood $\mathcal{U}_n$ of $J|_{A_n}\in \mathcal{J}(A_n,\omega)$ such that for every $J'\in \mathcal{J}(A_n,\omega,N)$, if $J'|_{A_n}\in \mathcal{U}_n$, then there is no embedded $J'$-holomorphic torus in $X$ intersecting $B_n$ with topological energy less than or equal to $E$. Assume the result does not hold, then there is a sequence of $J_n\subset \mathcal{J}(A,\omega,N)$ such that for each $n$ there exists a $J_n$-holomorphic map $\rho_n$ from a torus to $X$ which intersects $B$ and has topological energy less than or equal to $E$, and $J_n|_{A_n}\to J|_{A_n}$. 
For sufficiently large $n$, the distance between $\partial A_n$ and $B_n$ is greater than $M(N,E)$ with respect to the distance given by $J_n$, therefore the relevent $J_n$-holomorphic curve is contained in $A_n$. By Gromov's compactness theorem, a subsequence of $\rho_n$ will give a non-constant $J$-holomorphic map from a torus to $A_n$, such that the intersection $\textrm{Im}(\rho)\cap B$ is nonempty, and the topological energy of $\rho$ is less than or equal to $E$, which is a contradiction. Therefore, the claim holds. The result of the lemma then follows from part 1 of \ref{lem: diagonal argument}.
\end{proof}

With the preparations above, one can now give the proofs of lemma \ref{lem: Z2 admissibility of Ji} and lemma \ref{lem: Z2 admissibility of a family for taut foliation}.

\begin{proof}[Proof of lemma \ref{lem: Z2 admissibility of Ji}]
By the definition of the set $U_i$, the almost complex structure $J_i$ is $(d,E)$-regular for some constant $d>0$ with respect to $\overline{U_i}$. Apply lemma \ref{lem: Z2 admissibility} for $B=\overline{U_i}$, there is a perturbation $J_i'\in \mathcal{J}(X,\Omega_i,N)$ of $J_i$, such that $J_i'$ is $E$-admissible and $\sigma^*(J_i')=-J_i'$ on $\overline{U_i}$. Let $W_i$ be a small compact neighborhood of the union of lifts of Klein-bottle leaves such that $\sigma(W_i)=W_i$. The almost complex structure $J_i'$ can be taken to be equal to $J_i$ on $W_i$ since $J_i$ is already $(W_i,E)$-admissible. By the definition of the set $U_i$, every $J_i$-holomorphic map from a torus to $X$ is either a lift of Klein-bottle leaf or is mapped into the set $U_i$. Therefore lemma \ref{lem: existence of J curve is a closed condition} shows that when the perturbation is sufficiently small, every $J_i'$-holomorphic torus with homology class $e$ is either contained in $U_i$ or is contained in $W_i$. In the latter case the curve is contained in $\widetilde{Y}\times\{0\}$ and it is a lift of a Klein-bottle leaf of $\mathcal{L}_i$ in class $A$. Since $J_i'$ is $(d,E)$-regular with respect to $\overline{U_i}$, for every $J_i'$ holomorphic torus $C$ in $U_i$ one has $\sigma(C)\neq C$.
\end{proof}

\begin{proof}[Proof of lemma \ref{lem: Z2 admissibility of a family for taut foliation}]
The almost complex structures $J_0'$ and $J_1'$ can be connected by a smooth family of almost complex structures $J_s'\in \mathcal{J}(X, \Omega_s,J'_0,J'_1, N)$ such that $\sigma^*(J_s')=-J_s'$ on $V$. Use lemma \ref{lem: regularity and admissibility for family}
, the family $J_s'$ can be further perturbed to satisfy the desired conditions.
\end{proof}

\section{An example} \label{sec: an example}
This section gives an example of a taut foliation with an odd number of Klein-bottle leaves such that every closed leaf is nondegenerate. By corollary \ref{cor: existence of Klein bottle leaf}, every deformation of such a foliation via taut foliations has at least one Klein-bottle leaf.

Think of the torus $T_0=S^1\times S^1$ as a trivial $S^1$-bundle over $S^1$. Let $z_1,z_2\in\mathbb{R}/2\pi$ be the coordinates of the two $S^1$ factors, where $z_1$ is the coordinate for the fiber, and $z_2$ is the coordinate for the base. Let $\gamma$ be a closed curve on the base that wraps the $S^1$ once in the positive direction. Take a horizontal foliation $\hat{\mathcal{I}}$ on $T_0$ such that the holonomy along $\gamma$ has two fixed points: $z_1=0$ and $z_1=\pi$, and that holonomy map has nontrivial linearization at these two points. Moreover, choose $\hat{\mathcal{I}}$ so that it is invariant under the map $(z_1,z_2)\mapsto (z_1+\pi, \pi-z_2)$ and the map $(z_1,z_2)\mapsto (z_1,z_2+\pi)$.

Consider the pull back of the foliation $\hat{\mathcal{I}}$ to $T_0\times S^1$. Let $z_3\in\mathbb{R}/2\pi$ be the coordinate for the $S^1$ factor, then $\textrm{span}\{\hat{\mathcal{I}},\frac{\partial}{\partial z_3}\}$ defines a foliation $\mathcal{I}$ on $T_0\times S^1$. The foliation $\mathcal{I}$ is invariant under the maps 
\begin{align*}
\sigma_1:(z_1,z_2,z_3) &\mapsto (z_1+\pi, \pi-z_2, z_3)\\
\sigma_2:(z_1,z_2,z_3) &\mapsto (z_1, z_2+\pi, \pi-z_3)\\
\sigma_3:(z_1,z_2,z_3) &\mapsto (z_1+\pi, -\pi-z_2, \pi-z_3)
\end{align*}
The set $V=\{\textrm{id}, \sigma_1,\sigma_2,\sigma_3\}$ is a group acting freely and discountinuously on $T_0\times S^1$ and it preserves the coorientation of $\mathcal{I}$. The quotient foliation $\mathcal{I}/V$ has exactly one Klein-bottle leaf and it is nondegenerate. Therefore, one has the following result.
\begin{Proposition}
Every deformation of $\mathcal{I}/V$ through taut foliations must have at least one Klein-bottle leaf. \qed
\end{Proposition}

\bibliographystyle{amsplain}
\bibliography{references}

\end{document}